\documentclass[11pt]{article}

\usepackage[pagebackref,colorlinks=true,urlcolor=blue,linkcolor=red,citecolor=red]{hyperref}

\usepackage{epsfig, graphics, graphicx}
\usepackage{subcaption, comment, bm}
\usepackage[percent]{overpic}
\usepackage{float}
\usepackage{amssymb,bm}
\usepackage{amsthm}
\usepackage{color}
\usepackage[]{amsmath}
\usepackage[]{amsfonts}
\usepackage[]{fancyhdr}
\usepackage[]{graphicx,wrapfig}
\usepackage{enumitem}
\usepackage[utf8]{inputenc}
\usepackage[T1]{fontenc}
\usepackage{mathtools}
\usepackage{array}
\usepackage{comment}
\usepackage{amsmath}
\usepackage{tikz}
\usepackage{epigraph}
\usepackage{lipsum}
\allowdisplaybreaks

\def\th@plain{%
  \thm@notefont{}
  \itshape 
}
\def\th@definition{%
  \thm@notefont{}
  \normalfont 
}
\makeatother

\usepackage{srcltx} 
\usepackage{amscd}
\usepackage{amssymb}
\usepackage{amsmath}
\usepackage{pb-diagram}
\usepackage{graphicx}
\usepackage{hyperref}
\usepackage{array}
\usepackage[all]{xy}
\xyoption{matrix}
\xyoption{arrow}
\usepackage{pdfsync}
\usepackage{physics}
\usepackage{enumitem}
 
\usepackage{thmtools}
\usepackage{amssymb}

\DeclareUnicodeCharacter{2212}{-,+}
\usepackage[none]{hyphenat}
\theoremstyle{definition}
\newtheorem{definition}{Definition}[section]
\newtheorem{theorem}[definition]{Theorem}

\newtheorem{remark}[definition]{Remark}

\newcommand{\Ic}{\mathcal{I}}
\newcommand{\Jc}{\mathcal{J}}

\newcommand{\Lc}{\mathcal{L}}
\newcommand{\Mc}{\mathcal{M}}

\newcommand{\Rc}{\mathcal{R}}

\newcommand{\Tc}{\mathcal{T}}

\newcommand{\vx}{\textbf{\textit{x}}}

\newcommand{\vw}{\textbf{\textit{w}}}
\newcommand{\vv}{\textbf{\textit{v}}}

\newcommand{\vf}{\textbf{\textit{f}}}

\newcommand{\D}{\mathrm{d}}
\newcommand{\DR}{\mathbb{D}_R}
\newcommand{\vuu}{\textbf{\textit{u}}}
\title{\vspace{-1cm} Inversion of generalized V-line transforms of vector fields in $\mathbb{R}^2$}

\author{Rahul Bhardwaj\thanks{Department of Mathematics, Indian Institute of Technology, Jammu, J $\&$ K - 181221, India. \url{2020rma2092@iitjammu.ac.in}} \and Rohit Kumar Mishra\thanks{Department of Mathematics, Indian Institute of Technology, Gandhinagar, Gujarat - 382355, India. \url{rohit.m@iitgn.ac.in}, \url{rohittifr2011@gmail.com}} \and Manmohan Vashisth\thanks{Department of Mathematics, Indian Institute of Technology, Ropar, Punjab - 140001,  India. \url{manmohanvashisth@iitrpr.ac.in}}}

\begin{document}
\maketitle
\begin{abstract}
This article studies the inverse problem of recovering a vector field supported in $\DR$, the disk of radius $R$ centered at the origin, through a set of generalized broken ray/V-line transforms, namely longitudinal and transverse V-line transforms. Geometrically, we work with broken lines that start from the boundary of a disk and break at a fixed angle after traveling a distance along the diameter. We derive two inversion algorithms to recover a vector field in $\mathbb{R}^2$ from the knowledge of its longitudinal and transverse V-line transforms over two different subsets of aforementioned broken lines in $\mathbb{R}^2$.
\end{abstract}
\section{Introduction}
The question of recovering a scalar function, a vector field, or, more generally, a tensor field from its integral transforms of some kind is vital in the field of imaging sciences. Typically, the integral transforms appearing in the field of integral geometry consist of longitudinal, transverse, mixed, and momentum ray transforms. These transforms integrate certain weighted projections of functions (or, more generally, tensor fields) along straight lines or other trajectories in $\mathbb{R}^n$. The inverse problem of interest is to study these integral transforms and develop methods to reconstruct the unknown tensor field from its integral transforms. In the straight-line case, there is a vast literature available in various settings with different combinations of these transforms, see \cite{Denisyuk_1994, Denisyuk_2006, Derevtsov_2015, Katsevich_2006, Mishra_2020, Rohit_Suman_2021, Chandni_Rohit,  Monard1,  Palamodov_2009, Schuster_2000,  Sharafutdinov_Book, Sharafutdinov_2007, Mishra_2019} and references therein. 

In recent years, the problem of recovering a scalar function from its integration over broken lines or "V"-shaped lines has been studied by many authors. Such integrals over V-lines occur naturally in the field of optical tomography, where one uses light transmitted and scattered through an object to determine the interior features of that object. Often, one takes the measurements on the boundary of that object to reconstruct the spatially varying coefficients of light absorption and scattering. Under reasonable assumptions, it can be assumed that the majority of photons change their flight direction only once inside the object (see \cite{Florescu2009, Florescu2010, Florescu2011}), which motivates the name broken ray/V-line transform. The mathematical problem of recovering a scalar function from its V-line transform is highly overdetermined because the family of all V-lines in a plane is four-dimensional, while our unknown function depends only on two variables. Therefore, it is natural to expect that we can reconstruct the unknown function from a two-dimensional subset of V-line transform data. In literature, two classes of V-line transforms have been studied. The first class consists of V-lines with vertices on the boundary of the image domain (see \cite{Terzioglu_2018} and the references therein) and has applications in image reconstruction problems using Compton cameras. The second class includes V-lines with vertices inside the support of the scalar function, and these appear in the field of single scattering tomography mentioned above 
\cite{Ambarsoumian_2012, amb-lat_2019, Amb_Lat_star, Ambartsoumian_2013, ambartsoumian2016numerical, Florescu-Markel-Schotland, Gouia_Amb_V-line, Kats_Krylov-13, Sherson, walker2019broken, ZSM-star-14}. For a detailed discussion on these generalized operators and literature survey, please refer to a recent book by Ambartsoumian \cite{amb-book}. 

This article focuses on a generalization of the V-line transform defined for vector fields in $\mathbb{R}^2$. The V-line transforms for vector fields and symmetric 2-tensor fields in $\mathbb{R}^2$ have been studied in recent works \cite{ Gaik_Latifi_Rohit, Gaik_Mohammad_Rohit, Gaik_Rohit_Indrani}. These works extended the notions of straight-line longitudinal, transverse, and momentum ray transforms of vector fields and symmetric 2-tensor fields to those integrating along V-lines. The authors of these articles derived several exact inversion formulas to recover a vector field and symmetric 2-tensor field from certain combinations of the transforms mentioned above. In addition to the V-line transform, the article \cite{Gaik_Mohammad_Rohit} also extended the notion of the star transform from functions to vector fields in $\mathbb{R}^2$, and an inversion formula was also derived for this extended operator. Motivated by these works, we considered the same problem of recovering a vector field in $\mathbb{R}^2$ in a different geometric setting. We used the setup introduced in \cite{Ambarsoumian_2012, Ambartsoumian_2013} for the scalar case, where the collimated (directionally focused) emitters and receivers are located on the boundary of a disk. Our results can be considered an extension of \cite{Ambarsoumian_2012, Ambartsoumian_2013} in the vector field setting.

The article is organized as follows: In Section \ref{sec:Definition and notations}, we introduce the necessary notations and definitions of the integral transforms. Section \ref{sec:main results} is devoted to stating the main results of the article and some discussion about them. The proofs of the main theorems are presented in Sections \ref{sec:full recovery th} and \ref{sec:partial data case}. Finally, we conclude the article with acknowledgments in Section \ref{sec:acknowledgements}.
 \section{Definition and notations}\label{sec:Definition and notations}
This starting section is devoted to introducing notations and definitions used throughout this article. The regular fonts are used to represent scalars or scalar-valued functions (such as $x_1$, $x_2$, $f$, $g$, etc), and the bold fonts are used to represent vectors or vector fields in $\mathbb{R}^2$ (such as $\vf$, $\vx$, $\vv$, etc).

The disc of radius $R$ centered at the origin is represented by $\mathbb{D}_R$, and its boundary is denoted as $\partial \mathbb{D}_R$. Let $C_0^\infty(S^1(\DR))$ be the spaces of vector fields with components in $C_0^\infty(\DR)$, the space of smooth functions with compact support in $\DR$. The well-known differential operators, such as the gradient of a scalar function $\varphi(x,y)$, the divergence and the curl operators of vector fields $\vf:= (f_1,f_2)$ are defined as follows:
\begin{itemize}
\item The gradient and orthogonal gradient operators $\mathrm{d},\D^\perp :C_0^\infty(\DR) \rightarrow  C_0^\infty\left(S^1(\mathbb{D}_R)\right)$ are defined as follows:
$$  \D \varphi := \left(\frac{\partial{\varphi}}{\partial x_1},\frac{\partial{\varphi}}{\partial x_2}\right) \quad \mbox{ and } \quad \D^\perp \varphi := \left(-\frac{\partial{\varphi}}{\partial x_2},\frac{\partial{\varphi}}{\partial x_1}\right).$$
\item The divergence and the orthogonal divergence operators $\delta, \delta^\perp:C_0^\infty\left(S^1(\mathbb{D}_R)\right) \rightarrow  C_0^\infty\left(\mathbb{D}_R\right)$ are defined as follows:
$$\delta \vf := \frac{\partial{f_1}}{\partial x_1}+\frac{\partial{f_2}}{\partial x_2} \quad \mbox{ and } \quad  \delta^\perp \vf := \frac{\partial{f_2}}{\partial x_1} - \frac{\partial{f_1}}{\partial x_2}.$$
\end{itemize}
The operator $\delta^\perp$ is sometimes known as the 2-dimensional curl of vector fields.

Let $\vf$ be a vector field supported in the disk $\DR$ and $\theta \in (0, \pi/2)$ be a fixed angle. Following \cite{Ambarsoumian_2012}, we denote by $BR(\beta, d)$ the broken ray that starts from the point $\vx_\beta = (R \cos \beta, R \sin \beta)$ on $\partial \DR$ and moves the distance $d$ in the radial direction $\vuu_\beta = -(\cos \beta, \sin \beta)$, then breaks into another ray under the obtuse angle $\pi - \theta$ and travels in the direction $\vv_\beta = -\left(\cos (\theta + \beta), \sin (\theta +\beta)\right)$ (see Figure \ref{fig: def of broken ray}). More specifically, the broken ray $BR(\beta, d)$ is defined as follows:
\begin{align}\label{eq:definition of BR(beta,d)}
    BR(\beta,d) =  \left\{\vx_\beta + t \vuu_\beta: 0\leq t \leq d\right\}  \cup \left\{\vx_\beta + d \vuu_\beta + s \vv_\beta: 0 \leq s < \infty\right\}.
\end{align}
Note that the unit vectors $\vuu_\beta$ and $\vv_\beta$ are fixed and uniquely defined for a given $(\beta, d)$, which are used in the upcoming definitions.
\begin{figure}[H]
\centering
\begin{subfigure}[c]{0.48\textwidth}
\centering
\includegraphics[width=\textwidth]{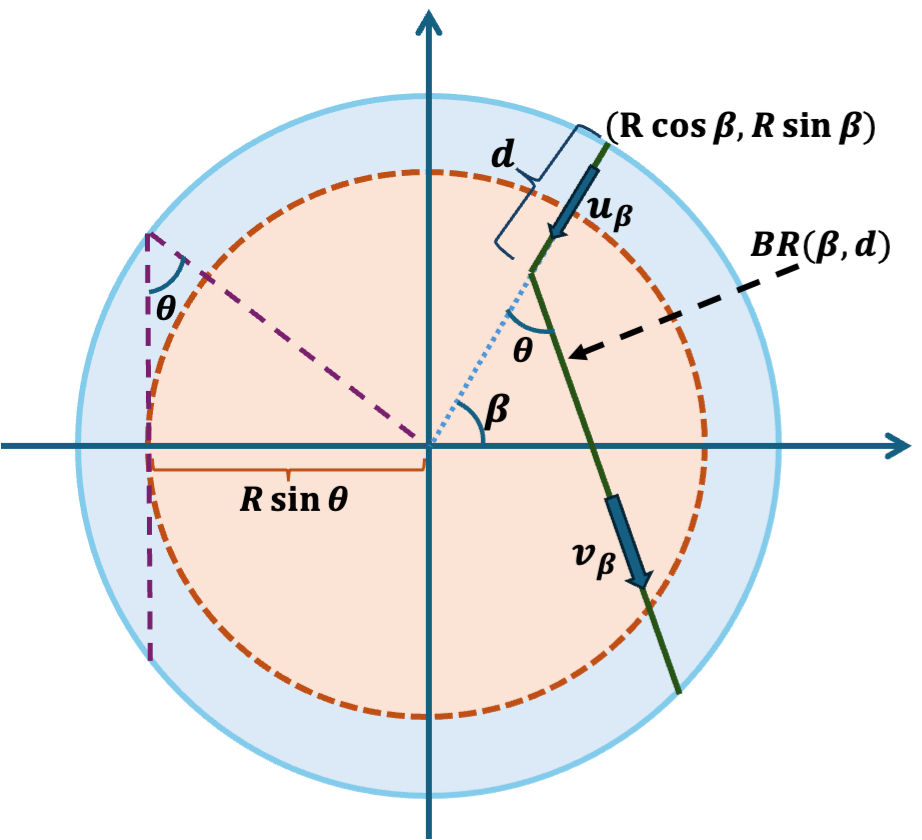}\caption{A sketch of broken line $BR(\beta, d)$.}\label{fig: def of broken ray}
\end{subfigure}
\hfill
\begin{subfigure}[c]{0.48\textwidth}
\centering
\includegraphics[width=\textwidth]{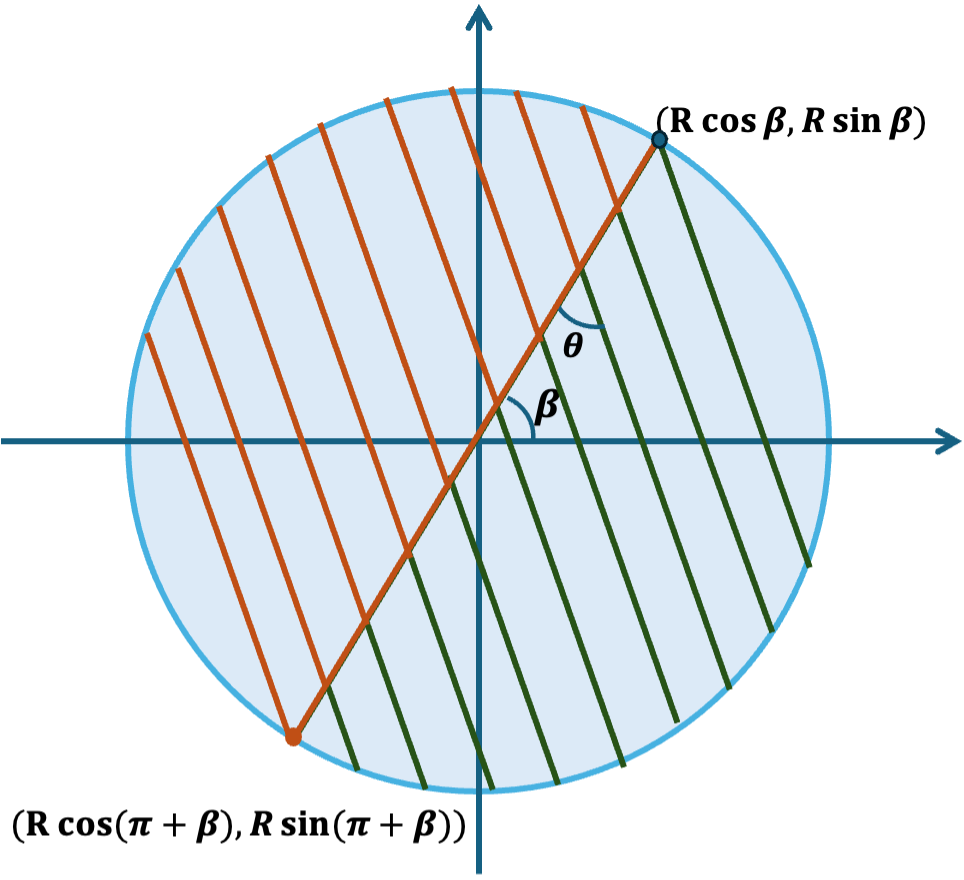}\caption{Broken lines considered to collect data.}\label{fig: data is taken over}
\end{subfigure}
\end{figure}
\begin{definition}\label{def:longitudinal v-line transform}
Let $\vf$ be a vector field with components $f_i \in C_0^\infty(\DR)$ for $ i = 1, 2$. The \textbf{longitudinal V-line transform} of $\vf$ is defined by
\begin{equation}\label{eq:def LVT}
\mathcal{L}\vf (\beta,d) := \int\limits_{0}^{d}\vuu_\beta \cdot\vf (\vx_\beta +s\vuu_\beta)\,ds + \int\limits_{0}^{\infty} \vv_\beta\cdot\vf(\vx_\beta + d\vuu_\beta + s\vv_\beta)\,ds, \quad \beta \in [0, 2\pi) \mbox { and } d \in [0,2R].
\end{equation}
\end{definition}
\begin{definition}\label{def:transverse v-line transform}
Let $\vf$ be a vector field with components $f_i \in C_0^\infty(\DR)$ for $ i = 1, 2$. The \textbf{transverse V-line transform} of $\vf$ is defined by
\begin{equation}\label{eq:def TVT}
\mathcal{T}\vf (\beta,d) := \int\limits_{0}^{d}\vuu_\beta^\perp \cdot\vf (\vx_\beta +s\vuu_\beta)\,ds + \int\limits_{0}^{\infty} \vv_\beta^\perp \cdot\vf(\vx_\beta + d\vuu_\beta + s\vv_\beta)\,ds, \quad \beta \in [0, 2\pi) \mbox { and } d \in [0,2R].
\end{equation}
For $\vuu_\beta =  -(\cos \beta, \sin \beta)$, we use the notation $\vuu^\perp = (\sin \beta, - \cos \beta)$. 
\end{definition}
\noindent Next, we define the straight-line version of these two transforms, which will be used later in the article. For given $\psi \in [0, 2\pi)$ and $p \in \mathbb{R}$, let $L(\psi,p):= \{(x_1,x_2): x_1 \cos{\psi} +x_2 \sin{\psi} = p\} $ be the line at a signed distance $p$ from the origin and normal to the unit vector $\vw = (\cos{\psi},\sin{\psi})$, and $\vw^\perp = (-\sin{\psi},\cos{\psi}) $ is a unit vector in  $\mathbb{R}^2$.  
\begin{definition}\label{def:longitudinal ray transform}
Let $\vf$ be a vector field with components $f_i \in C_0^\infty(\DR)$ for $ i = 1, 2$. The \textbf{longitudinal ray transform} of $\vf$ is defined by
\begin{equation}\label{eq:def LRT}
\mathcal{I}\vf (\psi,p) = \mathcal{I}\vf (\vw,p) := \int_\mathbb{R} \vw^\perp \cdot\vf (p \vw + s\vw^\perp)\,ds,  \quad \psi \in [0, 2\pi) \mbox { and } p \in \mathbb{R}.
\end{equation}
\end{definition}
\begin{definition}\label{def:transverse ray transform}
Let $\vf$ be a vector field with components $f_i \in C_0^\infty(\DR)$ for $ i = 1, 2$. The \textbf{transverse ray transform} of $\vf$ is defined by
\begin{equation}\label{eq:def TRT}
\mathcal{J}\vf (\psi,p) = \mathcal{J}\vf (\vw,p) := \int_\mathbb{R} \vw \cdot\vf (p \vw + s\vw^\perp)\,ds, \quad \psi \in [0, 2\pi) \mbox { and } p \in \mathbb{R}.
\end{equation}
\end{definition}
\begin{remark}
 The following identities are easy to verify:
$$ \Tc \vf =  - \Lc \vf^\perp \quad \mbox{ and } \quad \Jc \vf =  - \Ic \vf^\perp.$$
\end{remark}      
\noindent Finally, we also need the Radon transform and its inversion at a later stage, which we present now:
\begin{definition}\label{def:the Radon transform}
Let $f$ be a scalar function field in  $C_0^\infty(\DR)$. The \textbf{Radon transform} of $f$ is defined as follows
\begin{equation}\label{eq:def Radon transform}
\Rc f (\psi,p) = \Rc f (\vw,p) := \int_\mathbb{R} f (p \vw + s\vw^\perp)\,ds, \quad \psi \in [0, 2\pi) \mbox { and } p \in \mathbb{R}.
\end{equation}
\end{definition}
\noindent It is well-known that $f$ can be uniquely recovered from the knowledge of its Radon transform with the following explicit formula:
\begin{align}\label{eq:Radon's inversion formula}
    f(\vx) = \frac{1}{2\pi} \left(-\Delta\right)^{1/2} \int_0^{2\pi}\Rc f((\cos \alpha, \sin \alpha), x \cos \alpha + y \sin \alpha) d\alpha. 
\end{align}
\section{Main results}\label{sec:main results}
In this section, we present the main findings of this article. Both theorems presented here provide a method for reconstructing a vector field $\vf$, using the information of its longitudinal V-line transform and transverse V-line transform. These reconstruction methods are described in the proofs of the corresponding theorems; see equation \eqref{eq: Radon transforms of components} (which gives componentwise Radon transform of the unknown vector fields $\vf$) and equations \eqref{eq: Fourier coefficient of f1}, \eqref{eq: Fourier coefficient of f2} (which provide explicit formulas for Fourier coefficients of the components of $\vf$).
\begin{theorem}\label{th:full data recovery}
Let $\vf$ be a vector field with components in $C^{\infty}_0(\mathbb{R}^2)$ which are supported in $\mathbb{D}_{R\sin{\theta}}$. Then $\vf$ is uniquely determined from the knowledge of its  $\Lc \vf (\beta, d)$ and $\Tc \vf (\beta, d)$, for $d \in [0, 2R]$ and $\beta \in [0, 2\pi)$. 
\end{theorem}
\noindent Note that in this theorem, there is a restriction on the support of $\vf$, which depends on the fixed scattering angle $\theta$. This support condition is coming due to the technique we are using to prove this theorem. This theorem is proved by generating the straight-line transforms by combining the given V-line transform data in a particular way. It is clear from Figure \ref{fig: data is taken over} that straight-line transform can not be generated for lines outside the disk of radius $R\sin \theta$. This technique was introduced by Ambartsoumian \cite{Ambarsoumian_2012}, where he considered the same problem for the scalar functions. 
\begin{theorem}\label{th:partial data recovery}
Let $\vf \in C_0^\infty(S^{1}(\DR))$. Then $\vf$ is uniquely recovered from $\Lc \vf (\beta, d)$ and $\Tc \vf (\beta, d)$ which are known for $d \in [0, R]$ and $\beta \in [0, 2\pi)$. 
\end{theorem}
\noindent This theorem is more general in the sense that there is no restriction on the support of $\vf$ and the considering less data here in the sense that the scalar $d$ is varying in the half interval $[0, R]$ instead on $[0, 2R]$. Ambartsoumian and Moon have studied the same problem for the scalar field case in \cite{Ambartsoumian_2013}. The idea behind this theorem is to expand the data ($\Lc \vf \ \&\ \Tc\vf$) and the unknown vector field $\vf$ into their Fourier series and then try to express the Fourier coefficients of $\vf$ in terms of Fourier coefficients of $\Lc \vf$ and $\Tc\vf$.   
\section{Proof of Theorem \ref{th:full data recovery}}\label{sec:full recovery th}
In this section, we prove that the knowledge of longitudinal and transverse V-line transform uniquely determines the unknown vector field $\vf$. 
\begin{proof}
As discussed previously, we extend $\vf$ by zero outside $\mathbb{D}_{R\sin{\theta}}$ and denote the extended vector field again by $\vf$. Since $\vf$ is zero outside of the disc $\DR$. We start by noting that if we consider $d =2R$, then 
\begin{align}\label{eq:relation between VLT and regular RT}
    \Lc \vf(\beta, 2R) =  \Ic \vf(\beta + \pi/2, 0).
\end{align}
Let us consider
\begin{align*}
 \Lc\vf(\beta,d) = \underbrace{\int\limits_{0}^{d}\vuu_\beta \cdot\vf (\vx_\beta +s\vuu_\beta)\,ds}_{I_1} + \underbrace{\int\limits_{0}^{\infty} \vv_\beta\cdot\vf(\vx_\beta + d\vuu_\beta + s\vv_\beta)\,ds,}_{I_2}
\end{align*}
and 
\begin{align*}
\Lc\vf(\beta + \pi,2R-d) &= \underbrace{\int\limits_{0}^{2R - d}\vuu_{\beta + \pi} \cdot\vf (\vx_{\beta + \pi}  +s\vuu_{\beta + \pi} )\,ds}_{J_1} \\
&\qquad + \underbrace{\int\limits_{0}^{\infty} \vv_{\beta + \pi} \cdot\vf(\vx_{\beta + \pi}  + (2R - d)\vuu_{\beta + \pi}  + s\vv_{\beta + \pi}) \,ds}_{J_2}.
\end{align*}
Let us first simplify the term $J_1$.
\begin{align*}
    J_1 &= \int\limits_{0}^{2R - d}\vuu_{\beta + \pi} \cdot\vf (\vx_{\beta + \pi}  +s\vuu_{\beta + \pi} )\,ds\\
    &= \int\limits_{0}^{2R - d}\vuu_{\beta + \pi} \cdot\vf (\vx_{\beta + \pi} +2R \vuu_{\beta + \pi} - 2R \vuu_{\beta + \pi}  + s\vuu_{\beta + \pi} )\,ds\\
     &= \int\limits_{0}^{2R - d}\vuu_{\beta + \pi} \cdot\vf (\vx_{\beta} + (s - 2R) \vuu_{\beta + \pi})\,ds, \quad \mbox{ since } \vx_{\beta + \pi} +2R \vuu_{\beta + \pi} =  \vx_\beta\\
    &= \int\limits_{-2R}^{- d}\vuu_{\beta + \pi} \cdot\vf (\vx_{\beta} + s\vuu_{\beta + \pi})\,ds \\
    &= -\int\limits_{-2R}^{- d}\vuu_{\beta} \cdot\vf (\vx_{\beta} - s\vuu_{\beta})\,ds, \quad \mbox{ since }  \vuu_{\beta + \pi} =  - \vuu_\beta\\
    &= -\int\limits_{d}^{2R}\vuu_{\beta} \cdot\vf (\vx_{\beta} + s\vuu_{\beta})\,ds.
\end{align*}
Then, we have 
\begin{align*}
    I_1 - J_1 &= \int\limits_{0}^{d}\vuu_\beta \cdot\vf (\vx_\beta +s\vuu_\beta)\,ds + \int\limits_{d}^{2R}\vuu_{\beta} \cdot\vf (\vx_{\beta} + s\vuu_{\beta})\,ds =  \int\limits_{0}^{2R}\vuu_{\beta} \cdot\vf (\vx_{\beta} + s\vuu_{\beta})\,ds\\
    &= \Ic \vf \left(\beta + \pi/2, 0\right) =  \Lc\vf(\beta, 2R). 
\end{align*}
Repeating a similar calculation, we have the following identity:
\begin{align*}
    I_2 - J_2 = \mathcal{I}\vf\left(\psi_\beta, t_d \right), \quad \mbox { with } \psi_{\beta} = \beta + \theta + \pi/2 \ \mbox { and } \  t_{d} = (R-d)\sin(\pi + \theta).
\end{align*}
Using these above relations, we get the following relation:
\begin{align*}
\Lc\vf(\beta,d) -  \Lc\vf(\beta + \pi,2R-d) = \mathcal{I}\vf({\beta + \pi/2,0}) + \mathcal{I}\vf{{(\psi_{\beta},t_{d})}}
\end{align*}
The above relation implies 
\begin{align}\label{eq:LRT in terms of LVT}
\Ic\vf{{(\psi_{\beta},t_{d})}} = \Lc\vf(\beta,d) -  \Lc\vf(\beta + \pi,2R-d) - \Lc \vf(\beta, 2R).
\end{align}
Following a similar line of arguments, we get an analogous relation for transverse ray transform, which is given as follows:
\begin{align}\label{eq:TRT in terms of TVT}
\Jc\vf(\psi_{\beta},t_{d}) = \Tc\vf(\beta,d) -  \Tc\vf(\beta + \pi,2R-d) - \Tc\vf({\beta + \pi/2,0}).
\end{align}
Please note that the right-hand side of the above relations \eqref{eq:LRT in terms of LVT}, \eqref{eq:TRT in terms of TVT} is completely known in terms of given data. The left-hand sides are the longitudinal/transverse ray transform of $\vf$ along the line defined by the parameter $(\psi_\beta, t_d)$. Therefore, by varying the parameter $(\psi_\beta, t_d)$, we get the longitudinal/transverse ray transform of $\vf$ along every line passing through $\mathbb{D}_{R\sin \theta}.$ Once we know both longitudinal and transverse ray transforms of $\vf$, we can recover $\vf$ explicitly as presented in \cite[Section 3.2]{Derevtsov_2015}. For the sake of completeness, we briefly discuss the steps here.

As discussed above, we know the longitudinal/transverse ray transform of $\vf$ and start by rewriting this data using the definitions of respective transform as follows:
\begin{align*}
 \Ic \vf(\psi_\beta, t_d) &= - \sin \psi_\beta\  \Rc f_1 (\psi_\beta, t_d) + \cos \psi_\beta \  \Rc f_2 (\psi_\beta, t_d)\\
  \Tc \vf(\psi_\beta, t_d) &=  \cos \psi_\beta\  \Rc f_1 (\psi_\beta, t_d) + \sin \psi_\beta \  \Rc f_2 (\psi_\beta, t_d).
\end{align*}
By solving this system of equations, we have 
\begin{equation}\label{eq: Radon transforms of components}  
\begin{aligned}
\Rc f_1 (\psi_\beta, t_d) &= - \sin \psi_\beta\ \Ic \vf(\psi_\beta, t_d) + \cos \psi_\beta \  \Tc \vf (\psi_\beta, t_d)\\
\Rc f_2(\psi_\beta, t_d) &=  \cos \psi_\beta\  \Ic \vf(\psi_\beta, t_d) + \sin \psi_\beta \    \Tc \vf(\psi_\beta, t_d).
\end{aligned}
\end{equation}
Hence, we get the componentwise Radon transform $\vf = (f_1, f_2)$ and therefore, by applying the inversion formula \eqref{eq:Radon's inversion formula}, we obtain the $\vf$ explicitly, which completes the proof of the theorem. 
\end{proof}
\section{Proof of the Theorem \ref{th:partial data recovery}}\label{sec:partial data case}
We begin this section with a quick introduction to Mellin transform and some of its properties. 
\begin{definition}[\cite{Flajolet_1995}]
Let $f$ be an integrable function that decays at infinity. Then the Mellin transform for $f$ is denoted by $\mathcal{M}f$ and is defined by 
    \begin{equation*}
      \mathcal{M}f(s) := \int\limits_{0}^{\infty} p^{s-1}f(p) \,dp
  \end{equation*}
\end{definition}  
Here are some basic properties of the Mellin transform, which are crucial in proving our Theorem \ref{th:partial data recovery}. Let $f$ be an integrable function that decays at infinity, and then the following identities hold:
\begin{enumerate}
\item $\displaystyle \mathcal{M}\left[r^k f(r)\right](s) = \mathcal{M}f(s+k)\label{p1}$ 
\item $\displaystyle \mathcal{M}\left[\int\limits_{t}^{\infty}f(x)\,dx\right](s) = \frac{\mathcal{M}f(s+1)}{s}\label{p2} $ 
\end{enumerate}
\begin{figure}[H]
\centering
\begin{subfigure}[c]{0.46\textwidth}
\centering
\includegraphics[width=\textwidth]{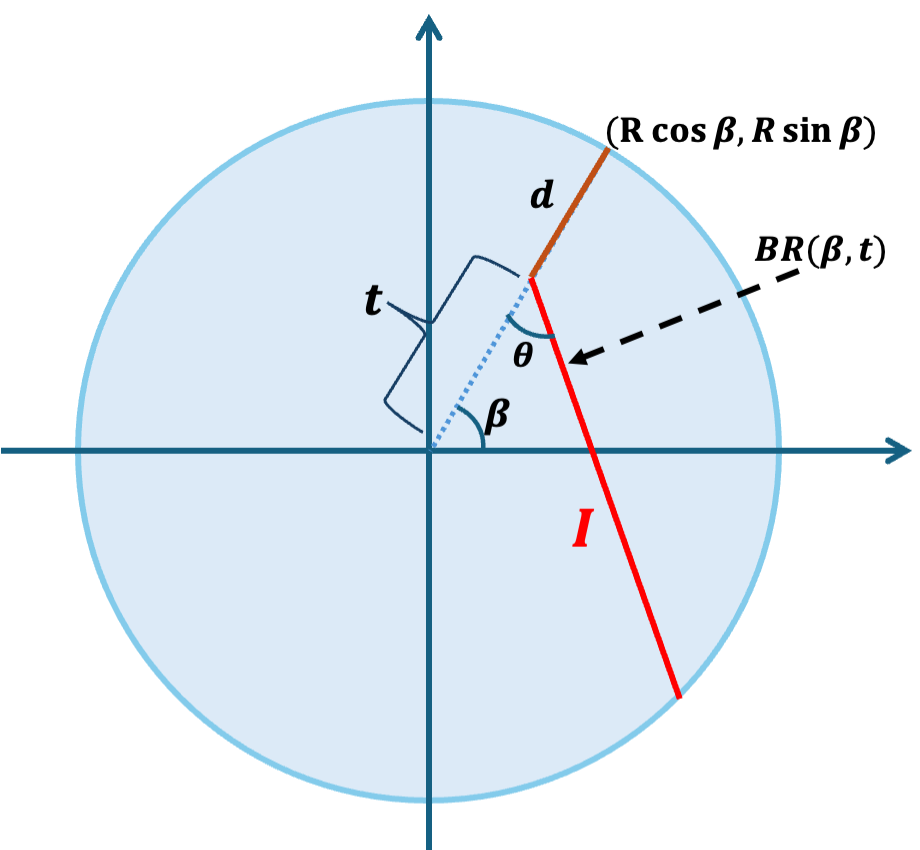}\caption{A sketch of modified parametrization of the broken line $B(\beta, t)$ with $ t = R- d$.}\label{fig: B(beta, t)}
\end{subfigure}
\hfill
\begin{subfigure}[c]{0.49\textwidth}
\centering
\includegraphics[width=\textwidth]{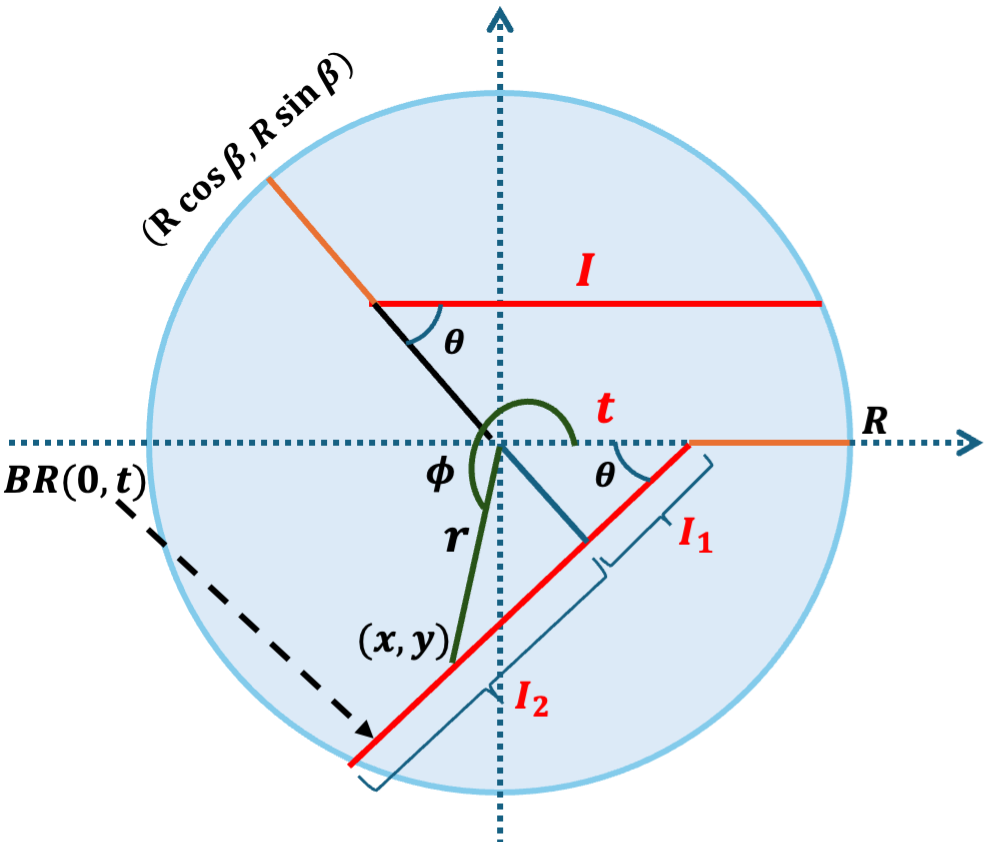}\caption{Here $B(0, t)$ is the broken line obtained by rotating $B(\beta, t)$ clockwise by an angle $\beta$.}\label{fig: polar coordinates}
\end{subfigure}
\end{figure}
To simplify the notations and upcoming calculations, we slightly change the parametrization of broken lines. Recall that the broken rays are defined using two parameters $\beta \in  [0, \pi)$ and {$d \in [0, R]$}, where $d$ is the distance traveled (along the diameter) by the ray before scattering. In this section, we change the parameter $d$ with a new parameter $t = R - d$; that is, from here on, the broken rays are parameterized by the ordered pair $(\beta,t)$, see Figure \ref{fig: B(beta, t)}.

Now, let us denote $p(\beta, t) := \mathcal{L}\vf(\beta, t)  $, $q(\beta, t) := \mathcal{T}\vf(\beta, t)  $ and let $\vf(\phi,r) := (f_1(\phi,r), f_2(\phi,r))$ be the unknown vector field in polar coordinates. Then the Fourier series of $f_1(\phi,r)$, $f_2(\phi,r)$, $q(\beta, t)$, and $p(\beta, t)$ concerning their angular variables with Fourier coefficients can be expressed as follows:
\begin{align}
     f_1(\phi,r) = \sum_{n=-\infty}^{\infty} a_n(r)e^{i n\phi}, \quad \mbox{ with } a_n(r) = \frac{1}{2\pi}\int\limits_{0}^{2\pi} f_1(\phi,r)e^{-in\phi}\,d\phi\label{eq:Fourier series of f1} \\
     f_2(\phi,r) = \sum_{n=-\infty}^{\infty} b_n(r)e^{i n\phi}, \ \quad \mbox{ with }  b_n(r) = \frac{1}{2\pi}\int\limits_{0}^{2\pi} f_2(\phi,r)e^{-in\phi}\,d\phi\label{eq:Fourier series of f2} \\
      p(\beta, t) = \sum_{n=-\infty}^{\infty} p_n(t) e^{in\beta}, \ \ \quad \mbox{ with } p_n(t) = \frac{1}{2\pi}\int\limits_{0}^{2\pi} p(\beta, t)e^{-in\beta}\,d\beta \label{eq:Fourier series of LVT}\\
    q(\beta, t) = \sum_{n=-\infty}^{\infty} q_n(t) e^{in\beta}, \ \ \quad \mbox{ with } q_n(t) = \frac{1}{2\pi}\int\limits_{0}^{2\pi} q(\beta, t)e^{-in\beta}\,d\beta \label{eq:Fourier series of TVT}.
\end{align}
Recall our aim is to recover $\vf$ from the knowledge of $\Lc \vf$ and $\Tc \vf$. The idea here is to express Fourier coefficients $a_n$ and $b_n$ in terms of Fourier coefficients $p_n$ and $q_n$. To achieve this, we first prove the Mellin transform $a_n$ and $b_n$ can be explicitly expressed in terms of the Mellin transform of $p_n$ and $q_n$. Finally, $a_n$ and $b_n$ are recovered by inverting the Mellin transform. 
\begin{theorem}\label{th: Melin transforms for Fourier coefficients}
For $n \in \mathbb{Z}$, let $a_n$ and $b_n$ be the Fourier coefficients (defined in \eqref{eq:Fourier series of f1} and \eqref{eq:Fourier series of f2}) of components of a vector field $\displaystyle \vf = (f_1, f_2) \in C_0^\infty\left(S^1(\DR)\right)$. Then, we have
\begin{equation*}
\mathcal{M}a_n(s) = \frac{-\Mc \left( p_{(n+1)} + p_{(n-1)}\right)(s-1) + i \Mc \left(q_{(n+1)} - q_{(n-1)}\right)(s-1) }{2 \left[\frac{1}{s-1} + \mathcal{M}h_n(s-1)\right]}, \hspace{.6cm} Re(s) > 1
\end{equation*} 
and 
\begin{equation*}
\mathcal{M}b_n(s) = \frac{-\Mc \left( q_{(n+1)} + q_{(n-1)}\right)(s-1) - i \Mc \left(p_{(n+1)} - p_{(n-1)}\right)(s-1) }{2 \left[\frac{1}{s-1} + \mathcal{M}h_n(s-1)\right]}, \hspace{.6cm} Re(s) > 1
\end{equation*} 
where $h_n$ is defined as follows
\begin{equation*}
h_{n}(t) = \begin{cases}
(-1)^n e^{i\theta} e^{in{\psi(t)}}\frac{1+t\cos[\psi(t)] + t^2 \sin{[\psi(t)]} \frac{\sin{\theta}}{\sqrt{1-t^2 \sin^2(\theta)}}}{\sqrt{1+t^2 + 2t\cos{[\psi(t)]}}} & : 0< t \leq 1\\
(-1)^n e^{i\theta}e^{in{\psi(t)}}\frac{1+t\cos[\psi(t)] + t^2 \sin{[\psi(t)]} \frac{\sin{\theta}}{\sqrt{1-t^2 \sin^2(\theta)}}}{\sqrt{1+t^2 + 2t\cos{[\psi(t)]}}}\\ 
\quad - e^{i\theta} e^{in[2\theta - {\psi(t)}]}\frac{1-t\cos[2\theta - \psi(t)] + t^2 \sin{[2\theta - \psi(t)]} \frac{\sin{\theta}}{\sqrt{1-t^2 \sin^2(\theta)}}}{\sqrt{1+t^2 - 2t\cos{[2\theta - \psi(t)]}}}   &  : 1<t < \frac{1}{\sin{\theta}}  \\
0 & : t > \frac{1}{\sin{\theta}}
\end{cases}
\end{equation*}
with $\psi(t) = \sin^{-1}{(t\sin{\theta})} + \theta$. 
\end{theorem}

\begin{proof} We start our analysis by first establishing relations between these Fourier coefficients $a_n$, $b_n$, $p_n$, and $q_n$, which will be later used in Theorem \ref{th: Melin transforms for Fourier coefficients}. Let $\Theta(\beta, t)$ be the unit vector field along the Broken-ray $BR(\beta,t)$. More specifically,  we have $\Theta(\beta, t) = \vuu_\beta = -(\cos \beta, \sin \beta)$ for the branch from the boundary $\partial \DR$ and $\Theta(\beta, t) = \vv_\beta = -(\cos (\theta + \beta), \sin (\theta +\beta))$. 
Consider
\begin{align*}
p(\beta,t) &=\Lc \vf(\beta,t)\\ &= \int\limits_{BR(\beta,t)} \vf(\phi,r)\cdot{\Theta}\,ds\\
&= \int\limits_{BR(0,t)} \vf\left(\phi+\beta,r \right)\cdot{\Theta}\,ds\\
&= \int_t^R  \vf\left(\beta,r \right)\cdot(-\cos{\beta},-\sin{\beta})\,dr  + \int_I \vf\left(\phi + \beta,r \right)\cdot(-\cos{(\beta + \theta)},-\sin{(\beta + \theta)})\,ds \\
&= \sum_{n=-\infty}^{\infty}  \int\limits_{t}^{R} \left(a_n(r),b_n(r)\right)\cdot (-\cos{\beta},-\sin{\beta})e^{in \beta}\,dr \\ 
& \quad + \sum_{n=-\infty}^{\infty}\int\limits_{I} \left(a_n(r), b_n(r)\right)\cdot(-\cos{(\beta+\theta)},-\sin{(\beta+\theta)}) {e}^{in(\phi + \beta)}\,ds.
\end{align*}
Here in the fourth line, we use the fact the first integral is along $x$-axis (hence $\phi= 0$ and area element is $dr$), and the second integral is along the other section of $B(0, t)$.  Now, using $\displaystyle \cos \theta = \frac{e^{i\theta} + e^{-i\theta}}{2}$ and  $\displaystyle \sin \theta = \frac{e^{i\theta} - e^{-i\theta}}{2i} $ in the above expression of $p(\beta, t)$, we have
\begin{align*}
-2 p(\beta, t) &=\sum_{n=-\infty}^{\infty}  \int\limits_{t}^{R} \left\{a_n(r)\left({e^{i\beta} + e^{-i\beta}}\right) - i b_n(r)\left({e^{i\beta} - e^{-i\beta}}\right)\right\}e^{in \beta}\,dr \\ 
& \quad + \sum_{n=-\infty}^{\infty}\int\limits_{I} \left\{a_n(r)\left({e^{i(\beta+ \theta)} + e^{-i(\beta + \theta)}}\right) - i b_n(r)\left({e^{i(\beta + \theta)} - e^{-i(\beta + \theta)}}\right) \right\}{e}^{in(\phi + \beta)}\,ds\\
&= \sum_{n=-\infty}^{\infty}\int\limits_{t}^{R} \left\{\left(a_n(r) - i b_n(r)\right)e^{i(n+1)\beta} +  \left(a_n(r) + i b_n(r)\right)e^{i(n-1)\beta}\right\}\,dr \\
& \quad + \sum_{n=-\infty}^{\infty}  \int\limits_{I} \left\{\left(a_n(r) - i b_n(r)\right)e^{in\phi}e^{i\theta}e^{i(n+1)\beta} + \left(a_n(r) + i b_n(r)\right)e^{in\phi}e^{i\theta}e^{i(n-1)\beta}\right\}\,ds\\
&= \sum_{n=-\infty}^{\infty}\int\limits_{t}^{R} \left\{\left(a_{n-1}(r) - i b_{n-1}(r)\right) +  \left(a_{n+1}(r) + i b_{n+1}(r)\right)\right\}e^{in\beta}\,dr \\
& \quad + \sum_{n=-\infty}^{\infty}  \int\limits_{I} \left\{\left(a_{n-1}(r) - i b_{n-1}(r)\right)e^{i(n-1)\phi} +  \left(a_{n+1}(r) + i b_{n+1}(r)\right)e^{i(n+1)\phi}\right\}e^{i\theta}e^{in\beta}\,ds.
\end{align*}
Now comparing the above relation with the $\displaystyle p(\beta, t)= \sum_{n=-\infty}^{\infty} p_n(t) e^{in\beta}$, we get 
 \begin{align}\label{eq:p_n}
-2 p_n(t)   &= \int\limits_{t}^{R} \left\{\left(a_{n-1}(r) - i b_{n-1}(r)\right) + \left(a_{n+1}(r) + i b_{n+1}(r)\right)\right\}\,dr \nonumber\\
& \quad + \int\limits_{I} \left\{\left(a_{n-1}(r) - i b_{n-1}(r)\right)e^{i(n-1)\phi} + \left(a_{n+1}(r) + i b_{n+1}(r)\right)e^{i(n+1)\phi}\right\}e^{i\theta}\,ds.
 \end{align}
Also
\begin{align}\label{eq5.4}
q(\beta,t) &=\mathcal{T} \vf(\beta,t)= -\mathcal{L} \vf^{\perp}(\beta,t) = -\int\limits_{BR(\beta,t)} \vf^{\perp}(\phi,r)\cdot \Theta\,ds = -\int\limits_{BR(0,t)} \vf^{\perp}\left(\phi+\beta,r \right)\cdot\Theta\,ds \nonumber\\ 
&= \sum_{n=-\infty}^{\infty}\int\limits_{t}^{R}  \left(b_n(r),-a_n(r)\right)\cdot (-\cos{\beta},-\sin{\beta})e^{in  \beta}\,dr \nonumber \\ 
& \qquad + \sum_{n=-\infty}^{\infty}\int\limits_{I} \left(b_n(r), -a_n(r)\right)\cdot(-\cos{(\beta+\theta)},-\sin{(\beta+\theta)}) {e}^{in(\phi + \beta)}\,ds \nonumber\\
-2q(\beta,t)&=  \sum_{n=-\infty}^{\infty} \int\limits_{t}^{R} \left\{ \left(b_{n-1}(r) + i a_{n-1}(r)\right) +  \left(b_{n+1}(r) - i a_{n+1}(r)\right)\right\}e^{in\beta}\,dr \nonumber\\
& \quad + \sum_{n=-\infty}^{\infty} \int\limits_{I} \left\{\left(b_{n-1}(r) + i a_{n-1}(r)\right)e^{i(n-1)\phi} + \left(b_{n+1}(r) - i a_{n+1}(r)\right)e^{i(n+1)\phi}\right\}e^{i\theta}e^{in\beta}\,ds.
\end{align}
Again comparing the above expression  with the Fourier series of $q(\beta, t)$, we get
\begin{align}\label{eq:q_n}
-2 q_n(t) &= \int\limits_{t}^{R} \left\{\left(b_{n-1}(r) + i a_{n-1}(r)\right) + \left(b_{n+1}(r) - i a_{n+1}(r)\right)\right\}\,dr \nonumber\\
& \quad + \int\limits_{I} \left\{\left(b_{n-1}(r) + i a_{n-1}(r)\right)e^{i(n-1)\phi} + \left(b_{n+1}(r) - i a_{n+1}(r)\right)e^{i(n+1)\phi}\right\}e^{i\theta}\,ds.
 \end{align}
Multiply equation \eqref{eq:q_n} with $i$ and add it to equation \eqref{eq:p_n} to obtain
\begin{align}\label{eq:p_n + iq_n}
- \left(p_n(t) + i q_n(t) \right) &= \int\limits_{t}^{R} \left(a_{n+1}(r) + i b_{n+1}(r)\right)\,dr + \int\limits_{I} \left(a_{n+1}(r) + i b_{n+1}(r)\right)e^{i(n+1)\phi}e^{i\theta}\,ds.
\end{align}
Now we multiply equation \eqref{eq:q_n} with $i$ and subtract from equation \eqref{eq:p_n} to get
\begin{align}\label{eq:p_n - iq_n}
- \left(p_n(t) - i q_n(t) \right) &= \int\limits_{t}^{R} \left(a_{n-1}(r) - i b_{n-1}(r)\right)\,dr  + \int\limits_{I} \left(a_{n-1}(r) - i b_{n-1}(r)\right)e^{i(n-1)\phi}e^{i\theta}\,ds.
\end{align}
Equations \eqref{eq:p_n + iq_n} and \eqref{eq:p_n - iq_n} can be further rewritten in the following form:
\begin{align}
- \left(p_{(n-1)}(t) + i q_{(n-1)}(t) \right) &= \int\limits_{t}^{R} \left(a_{n}(r) + i b_{n}(r)\right)\,dr + \int\limits_{I} \left(a_{n}(r) + i b_{n}(r)\right)e^{in\phi}e^{i\theta}\,ds
\end{align}
and
\begin{align}
- \left(p_{(n+1)}(t) - i q_{(n+1)}(t) \right) &= \int\limits_{t}^{R} \left(a_{n}(r) - i b_{n}(r)\right)\,dr  + \int\limits_{I} \left(a_{n}(r) - i b_{n}(r)\right)e^{in\phi}e^{i\theta}\,ds.
\end{align}
Further simplification of the above two equations gives
\begin{align}
\frac{- \left(p_{(n+1)}(t) + p_{(n-1)}(t) \right)+ i\left(q_{(n+1)}(t) - q_{(n-1)}(t) \right)}{2} = \int\limits_{t}^{R} a_{n}(r) \,dr + \int\limits_{I} a_{n}(r) e^{in\phi}e^{i\theta}\,ds \label{eq:relation between an, pn, and qn} \\
\frac{\left(p_{(n+1)}(t) - p_{(n-1)}(t) \right) - i\left(q_{(n+1)}(t) + q_{(n-1)}(t) \right)}{2i} =  \int\limits_{t}^{R} b_{n}(r) \,dr + \int\limits_{I} b_{n}(r) e^{in\phi}e^{i\theta}\,ds.\label{eq:relation between bn, pn, and qn} 
\end{align}
To further simplify \eqref{eq:relation between an, pn, and qn} and \eqref{eq:relation between bn, pn, and qn}, we divide the line segment $I$ in two parts $I_1$ and $I_2$, with $ I_1 := \{(x,y) : y = (x-t)\tan{\theta},~ t\sin^2{\theta} \leq x \leq t \}$ and  $ I_2 := \{(x,y) : y = (x-t)\tan{\theta},~ -\infty \leq x \leq t\sin^2{\theta} \}$ (see Figure \ref{fig: B(beta, t)}). The polar angles made by a point $(x,y)\in I$ is given by 
\begin{align*}
    \phi(r) = \left\{\begin{array}{cc}
     2 \theta -  \psi(t/r)  ,     & (x, y) \in I_1 \\
     \pi +   \psi(t/r) ,   & (x, y) \in I_2
    \end{array},\right. 
\end{align*}
where $\psi(t) = \sin^{-1}\left(t \sin \theta\right) +  \theta$.
The length measures $ds$ on $I$ is given by (please refer \cite[Theorem 2]{Ambartsoumian_2013} for details)
\begin{align*}
    ds = \frac{r - t\cos{\phi} + tr \frac{d\phi}{dr}\sin{\phi}}{\sqrt{{r}^2 + t^2 - 2{r} t \cos{\phi}}}\,dr.
\end{align*}
Hence we have
\begin{align*}
ds = \begin{cases}
\frac{r-t\cos[2\theta - \psi(\frac{t}{r})] + \frac{t^2}{r} \sin{[2\theta - \psi(\frac{t}{r})]} \frac{\sin{\theta}}{\sqrt{1-\frac{t^2}{r^2} \sin^2\theta}}}{\sqrt{r^2+t^2 - 2tr\cos{[2\theta - \psi(\frac{t}{r})]}}}dr,  & \quad \mbox{ for } (x,y) \in I_1\\
\frac{r - t\cos[\pi + \psi(\frac{t}{r})] - \frac{t^2}{r} \sin{[\pi + \psi(\frac{t}{r})]} \frac{\sin{\theta}}{\sqrt{1-\frac{t^2}{r^2} \sin^2\theta}}}{\sqrt{r^2+t^2 - 2tr\cos{[\pi + \psi(\frac{t}{r})]}}}dr, & \quad \mbox{ for } (x,y) \in I_2.
\end{cases}
\end{align*}
Substituting the value of $\phi$ and $ds$ into expression \eqref{eq:relation between an, pn, and qn}, we have
\begin{align*}
\frac{1}{2}&\left[- \left(p_{(n+1)}(t) + p_{(n-1)}(t) \right)+ i\left(q_{(n+1)}(t) - q_{(n-1)}(t) \right) \right]\\&= \int\limits_{t}^{R} a_n(r)\,dr + \int\limits_{t}^{t\sin{\theta}}a_n(r) e^{i\theta} e^{in[2\theta - {\psi(\frac{t}{r})}]}\frac{r-t\cos[2\theta - \psi(\frac{t}{r})] + \frac{t^2}{r} \sin{[2\theta - \psi(\frac{t}{r})]} \frac{\sin{\theta}}{\sqrt{1-\frac{t^2}{r^2} \sin^2\theta}}}{\sqrt{r^2+t^2 - 2tr\cos{[2\theta - \psi(\frac{t}{r})]}}}\,dr\\ & \quad + \int\limits_{t\sin{\theta}}^{\infty}a_n(r) e^{i\theta} e^{in[\pi + {\psi(\frac{t}{r})}]}\frac{r - t\cos[\pi + \psi(\frac{t}{r})] - \frac{t^2}{r} \sin{[\pi + \psi(\frac{t}{r})]} \frac{\sin{\theta}}{\sqrt{1-\frac{t^2}{r^2} \sin^2\theta}}}{\sqrt{r^2+t^2 - 2tr\cos{[\pi + \psi(\frac{t}{r})]}}}\,dr \\     
 &= \int\limits_{t}^{R} a_n(r)\,dr - \int\limits_{t\sin{\theta}}^{t}a_n(r) e^{i\theta} e^{in[2\theta - {\psi(\frac{t}{r})}]}\frac{1-\frac{t}{r}\cos[2\theta - \psi(\frac{t}{r})] + \frac{t^2}{r^2} \sin{[2\theta - \psi(\frac{t}{r})]} \frac{\sin{\theta}}{\sqrt{1-\frac{t^2}{r^2} \sin^2\theta}}}{\sqrt{1+\frac{t^2}{r^2} - 2\frac{t}{r}\cos{[2\theta - \psi(\frac{t}{r})]}}}\,dr\\ & \quad + (-1)^n\int\limits_{t\sin{\theta}}^{\infty}a_n(r) e^{i\theta} e^{in[{\psi(\frac{t}{r})}]}\frac{1 + \frac{t}{r}\cos[\psi(\frac{t}{r})] + \frac{t^2}{r^2} \sin{[\psi(\frac{t}{r})]} \frac{\sin{\theta}}{\sqrt{1-\frac{t^2}{r^2} \sin^2\theta}}}{\sqrt{1+\frac{t^2}{r^2} + 2\frac{t}{r}\cos{[\psi(\frac{t}{r})]}}}\,dr \\ 
&= \int\limits_{t}^{R} a_n(r)\,dr + \{[r a_n(r)]\times h_n\}(t) 
\end{align*}
where
\begin{equation*}
\{g\times h\}(s) = \int\limits_{0}^{\infty} g(r)h\left(\frac{s}{r}\right)\frac{dr}{r}.
\end{equation*}
Now taking the Mellin transform above on both sides and using the  properties \eqref{p1} and \eqref{p2} of the Mellin transform, we get
\begin{equation*}
-\Mc \left( p_{(n+1)} + p_{(n-1)}\right)(s) + i \Mc \left(q_{(n+1)} - q_{(n-1)}\right)(s)  = 2  \left[\frac{1}{s}\Mc a_n(s+1) + \Mc a_n(s+1)\Mc h_n(s)\right].
\end{equation*}
Hence  we get 
\begin{equation*}
\mathcal{M}a_n(s) = \frac{-\Mc \left( p_{(n+1)} + p_{(n-1)}\right)(s-1) + i \Mc \left(q_{(n+1)} - q_{(n-1)}\right)(s-1) }{2 \left[\frac{1}{s-1} + \mathcal{M}h_n(s-1)\right]}, \hspace{.6cm} Re(s) > 1.
\end{equation*} 
Following the same analysis, we also obtain the Mellin transform of $b_n$, which is given by 
\begin{equation*}
\mathcal{M}b_n(s) = \frac{-\Mc \left( q_{(n+1)} + q_{(n-1)}\right)(s-1) - i \Mc \left(p_{(n+1)} - p_{(n-1)}\right)(s-1) }{2 \left[\frac{1}{s-1} + \mathcal{M}h_n(s-1)\right]}, \hspace{.6cm} Re(s) > 1.
\end{equation*} 
This completes the proof of the Theorem \ref{th: Melin transforms for Fourier coefficients}.
\end{proof}
Finally, we conclude this section by taking the inverse of the Mellin transform to get the following expressions for $a_n$ and $b_n$:
    \begin{equation}\label{eq: Fourier coefficient of f1}
        a_n(r) = \lim_{T\to\infty} \frac{1}{2\pi i} \int\limits_{t-Ti}^{t+Ti} r^{-s} \frac{-\Mc \left( p_{(n+1)} + p_{(n-1)}\right)(s-1) + i \Mc \left(q_{(n+1)} - q_{(n-1)}\right)(s-1) }{2 \left[\frac{1}{s-1} + \mathcal{M}h_n(s-1)\right]}\,ds
    \end{equation}
    and
   \begin{equation}\label{eq: Fourier coefficient of f2}
        b_n(r) = \lim_{T\to\infty} \frac{1}{2\pi i} \int\limits_{t-Ti}^{t+Ti} r^{-s} \frac{-\Mc \left( q_{(n+1)} + q_{(n-1)}\right)(s-1) - i \Mc \left(p_{(n+1)} - p_{(n-1)}\right)(s-1) }{2 \left[\frac{1}{s-1} + \mathcal{M}h_n(s-1)\right]}\,ds,
    \end{equation}
which completes the proof of Theorem \ref{th:partial data recovery}.
\section*{Acknowledgements}\label{sec:acknowledgements}
RB acknowledges the support of UGC, the Government of India, with a research fellowship. RM was partially supported by SERB SRG grant No. SRG/2022/000947. 

\bibliographystyle{plain}
\bibliography{references}

\end{document}